\documentclass[12pt]{amsart}
\usepackage{graphicx}

\usepackage{upref,amsxtra,amssymb,amscd}
\usepackage{varioref}
\usepackage{verbatim}
\usepackage{epsfig}
\usepackage{color}
\usepackage{mathrsfs}

\usepackage{latexsym}
\usepackage{epic,eepic,eucal}
\usepackage{enumerate}

\def\mE{					\mathcal E}

\def\t{           \theta}

\def\a{         \alpha}
\def\b{         \beta}
\def\g{         \gamma}

\newcommand{\NN}{{\mathbb N}}
\newcommand{\RR}{{\mathbb R}}

\newtheorem{theo}{\sc Theorem}[section]

\newtheorem{lemm}[theo]{\sc Lemma}
\newtheorem{coro}[theo]{\sc Corollary}

\theoremstyle{definition}

\theoremstyle{remark}

\newtheorem{rema}[theo]{\sc Remark}

\numberwithin{equation}{section}

\begin{document}





\title[Bounded L{\"u}roth expansions and Schmidt games] {Bounded L{\"u}roth expansions: applying Schmidt games where infinite distortion exists}

\author[B. Mance]{Bill Mance}
\author[J. Tseng]{Jimmy Tseng}
\address{B. M.: Department of Mathematics, Ohio State University, Columbus, OH 43210, USA}
\email{mance.8@osu.edu}
\address{J. T.: Department of Mathematics, University of Illinois at Urbana-Champaign, 61801, USA}
\email{tsengj@illinois.edu}

\begin{abstract}
We show that the set of numbers with bounded L{\"u}roth expansions (or bounded L{\"u}roth series) is winning and strong winning.  From either winning property, it immediately follows that the set is dense, has full Hausdorff dimension, and satisfies a countable intersection property.  Our result matches the well-known analogous result for bounded continued fraction expansions or, equivalently, badly approximable numbers. 

We note that L{\"u}roth expansions have a countably infinite Markov partition, which leads to the notion of infinite distortion (in the sense of Markov partitions).
\end{abstract}

\date{}
\maketitle

\section{Introduction}\label{secIntro}

In the fields of number theory and dynamical systems, a type of largeness of a set, that of being winning in the sense of Schmidt games, has quickly become important because fundamental sets in number theory---badly approximable numbers~\cite{Sch2}, badly approximable matrices~\cite{Sch3}---and fundamental sets in dynamical systems---points with nondense orbits under $C^2$-expanding circle maps~\cite{T4}, points with nondense orbits under linear automorphisms of the torus~\cite{BFK}---exhibit the winning property while not exhibiting the usual notion of largeness, namely that of being full Lebesgue measure.  These sets are, moreover, null sets, but they do satisfy a countable intersection property and have full Hausdorff dimension, much like conull sets.  

The main technique in this paper, the use of the notion of commensurate, is an advance in the application of the technique of Schmidt games to cases of infinite distortion (in the sense of Markov partitions), whereas previously only bounded distortion could be handled.\footnote{We defer the discussion of distortion to Section~\ref{secConcl}, especially Section~\ref{secInfVsFinDistort2}.}  The technique in this paper, which handles this new, extra source of infinity, is a significant extension of the second-named author's work in~\cite{T4} for some bounded distortion cases.  We will borrow notation and ideas from~\cite{T4}.  Moreover, we intend our proof in this paper, which shows that the set of bounded L{\"u}roth expansions is winning, to be a model for other cases which have infinite distortion, and such cases are plentiful and natural in number theory and dynamical systems.\footnote{Irrational numbers with bounded continued fraction expansions or, equivalently, bounded Gauss map expansions are already known to be winning.  The two known proofs rely on the fact that the set of these expansions is exactly the set of badly approximable numbers.  This equivalence is not available for L{\"u}roth expansions, and, therefore, the proof given in this paper is the only known proof that the set of bounded L{\"u}roth expansions is winning.  See Section~\ref{subsecFA} for more details.}

\subsection{Introduction to L{\"u}roth expansions} Perhaps the simplest example of a case with infinite distortion is that of L{\"u}roth expansions.  These expansions are variants of the well-known continued fraction expansions and $n$-ary expansions where $n>1$ is an integer.  Like for these other expansions, every real number $x \in X := [0,1] / 0 \sim 1$ can be written as a finite or infinite series, called the \textit{L{\"u}roth expansion} or, equivalently, \textit{L{\"u}roth series} of $x$: \begin{align*}x = \frac 1 {a_1(x)} & +  \frac 1 {a_1(x) (a_1(x) -1) a_2(x)} + \cdots \\ & + \frac 1 {a_1(x) (a_1(x) -1) \cdots a_{n-1}(x)(a_{n-1}(x)-1)a_n(x)}+\cdots\end{align*} where the natural number $a_k(x) \geq 2$ denotes the \textit{$k$-th digit} for each integer $k \geq 1$. Also like for these other expansions, the digits of the L{\"u}roth expansion for $x$ are given by a dynamical system, namely $T:X \rightarrow X$ where \begin{equation}\label{eqnLurothMap} Tx = \left\{\begin{array}{ll} n(n+1)x-n & \textrm{if } x \in [\frac 1{n+1}, \frac 1 n), \\
0 & \textrm{if } x=0.\end{array} \right. \end{equation} The first digit $a_1(x) = n+1$ if $x \in [\frac 1 {n+1}, \frac 1 {n})$ for $n \geq 1$, and the $k$-th digit is obtained by iterating the dynamical system: $a_k(x) = a_1(T^{k-1}x)$---when $T^{k-1} x =0$ for a $k \geq 1$, there is no digit and we stop the iteration, obtaining a finite expansion.\footnote{The number $0$ has no L{\"u}roth expansion according to our strict definition.}  A number $x$ has \textit{bounded L{\"u}roth expansion} if there exists a natural number $N(x)$ such that $a_k(x) \leq N(x)$ for all $k \geq 1$, and, in particular, finite L{\"u}roth expansions are bounded.  We discuss properties of L{\"u}roth expansions in Section~\ref{secPropLuroth}.  An introduction to L{\"u}roth expansions and related number-theoretical objects can be found in a number of references, including a short monograph~\cite{DK} by K.~Dajani and C.~Kraaikamp or another short monograph~\cite{Ga} by J.~Galambos.

\section{Statement of results}

As a model proof for applying Schmidt games to cases of infinite distortion, we show\footnote{Schmidt games and winning and strong winning sets are briefly discussed in Section~\ref{secSchmidtGames}; distortion, in Section~\ref{secInfVsFinDistort2}.}:

\begin{theo}\label{thmBndLurothWinning} The set of numbers with bounded L{\"u}roth expansion is $\a$-winning and $\a$-strong winning for $\a := \frac 1 8$.
\end{theo}

\noindent Let $T$ be the dynamical system in (\ref{eqnLurothMap}) that gives the digits of the L{\"u}roth expansion.  The theorem is equivalent (in a straightforward manner) to the following dynamical reinterpretation:

\begin{coro}
The set of points in $X$ whose forward orbits under $T$ miss an open interval with left endpoint $0$ is $\a$-winning and $\a$-strong winning for $\a := \frac 1 8$.
\end{coro}

\noindent Now the intersection of any other (strong) winning set in $\RR$ (or $X$) with the set of bounded L{\"u}roth expansions is also (strong) winning.  Using results from~\cite{Mc, BFK}, one possible number-theoretic corollary is

\begin{coro}
The set of badly approximable numbers with bounded L{\"u}roth expansion and bounded block $n$-ary expansion for every natural number $n >1$ is winning and strong winning and thus is a dense set of full Hausdorff dimension.
\end{coro}

\noindent A real number $x$ has \textit{bounded block-$k$} $n$-ary expansion if there exists an $N(x)>0$ such that every block of consecutive digits $k$ has at most $N(x)$ digits, and the number $x$ has \textit{bounded block} $n$-ary expansion if it has bounded block-$k$ $n$-ary expansion for all digits $k$.  There is an obvious dynamical reinterpretation of the corollary, an interpretation which allows us to replace the number-theoretical concept of numbers having bounded block $n$-ary expansion with the natural (and more general) dynamical concept of numbers whose forward orbits (under $x \mapsto n x \mod 1$) miss some neighborhood of a prescribed point. 

Our Theorem~\ref{thmBndLurothWinning} considerably strengthens the result of~\cite{SLZ2} (on the Hausdorff dimension of bounded L{\"u}roth expansions) to the winning properties and allows us to obtain the above corollaries, which are not obtainable just from knowing the Hausdorff dimension.  Furthermore, since the winning property is preserved by local isometries (see~\cite{Sch2}) and the strong winning property is preserved by quasisymmetric homeomorphisms (see~\cite{Mc}), these properties allow us to write corollaries similar to Corollary~1.2 of~\cite{T2}, corollaries which again are not obtainable just from knowing dimension.

\section{Properties of L{\"u}roth expansions}\label{secPropLuroth}

The elementary properties of L{\"u}roth expansions, we take from Chapter 2 of the monograph~\cite{DK}.  As typical with expansions such as these, we are only concerned with the numbers from the circle $X := [0,1] / 0 \sim 1$ because these expansions are defined modulo $1$.  In this section, a very important way of dealing with L{\"u}roth expansions, the key concept of commensurate, will be introduced.

Let us first, however, introduce some notation.  The absolute value of a set denotes the usual length or, equivalently, the probability Haar measure on $X$.  The \textit{L{\"u}roth element (LE) of generation $0$} is $X$.  For $n \in \NN$, \textit{a L{\"u}roth element (LE) of generation $n$} is a closed interval corresponding to the first $n$ digits in a L{\"u}roth expansion---this is the left-closed, right-open interval corresponding to all L{\"u}roth expansions beginning with the given $n$ digits union the right endpoint of this interval.  We refer to these $n$ digits as the L{\"u}roth expansion corresponding to this LE.  For $n \in \NN \cup \{0\}$, let $G_n$ denote the set of LEs of generation $n$. (Hence, $G_1 := \{[1/2, 1], [1/3, 1/2], \cdots, [1/k+1, 1/k], \cdots\}$.)

Given these notions of L{\"u}roth elements and generations, we observe the following properties:

\begin{enumerate}
	\item Let $n \in \NN$. Every element $E \in G_n$ has a unique left-adjacent element in $G_n$.  We denote this element by $E^-$.
	\item Let $n \in \NN$.  Every element $E \in G_n$ that does not correspond to a L{\"u}roth expansion ending in the digit $2$ has a unique right-adjacent element in $G_n$.  We denote this element by $E^+$.
	
	\item Given $n \in \NN \cup \{0\}$ and $E \in G_n$. Let $U$ be any open set containing the left endpoint of $E$. Then $U \cap E$ contains infinitely many elements of $G_{n+1}$.  
	\item Given $n \in \NN \cup \{0\}$ and $E \in G_n$.  The only point of $E$ satisfying the previous property is the left endpoint.
	\item The maximum over all diameters of elements in $G_n$ goes to zero as $n$ goes to infinity.
	\item Let $n \in \NN$.  For every $E \in G_n$ there exists an unique $F \in G_{n-1}$ such that $E \subset F$.
	\item Let $n \in \NN \cup \{0\}$. The interiors of any two distinct elements of $G_n$ are disjoint.

\end{enumerate}

Let $n \in \NN \cup \{0\}$ and $E \in G_n$ be chosen.  A point of $E$ is \textit{an accumulation point of $E$} if it satisfies the condition in Property (3).  Thus the left endpoint of any LE is an accumulation point of that element.  Property (4) says that it is the only one.  \textit{An accumulation point of generation $n$} is an accumulation point of some LE of generation $n$.  Note that the intersection of the sets of accumulation points of distinct generations is empty. Also note that the right endpoint of $E$ is an accumulation point of some generation less than or equal to $n$.

In the dynamical point of view, the accumulation points are the preimages of $0$ under $T$.  In particular, let $E$ be an LE of generation $n$.  Then the accumulation point belonging to $E$ is a preimage of $0$ under $T^n$, but not under $T^{n-1}$.  And every such preimage is an accumulation point for an LE of generation $n$.  

For us, a ball is assumed to have nonempty interior and it is usually assumed to be closed.  Let $n \in \NN$.  A closed ball $B$ is \textit{commensurate with generation $n$ (c.w.g $n$)} if $B$ completely contains an LE of generation $n$ but no LE of generation $n-1$.  Since every ball is path-connected, it follows that $B$ can contain at most one accumulation point of all generations up to (and including) $n-1$.

\begin{rema}
This notion of being commensurate---the essence of the proof of our main result---requires knowledge of both the length and position of $B$.  It is possible for $B$ to be arbitrarily small but commensurate with a small generation number.  Such $B$s must be avoided if we are to play the Schmidt game.
\end{rema}

The following lemma is easy to verify:

\begin{lemm}For every closed ball $B$ of $X$ that is also a proper subset of $X$, there exists a unique $n \in \NN$ such that $B$ is c.w.g. $n$.
\end{lemm}

\noindent The lemma has a corollary:
\begin{coro}\label{coroAtMost1AccPoint}
Any closed ball c.w.g. $n$ (and also properly contained in $X$) is properly contained in at most two elements of $G_{n-1}$.
\end{coro}

\begin{proof}
The ball $B$ is an interval and thus path-connected.  Proper containment follows by the definition of commensurate.  If $B$ contains three elements of $G_{n-1}$, then pick an interior point from each of these elements.  One of these points is closest to the left endpoint of $B$ and another is closest to the right endpoint of $B$.  The third point must lie between the other two.  As these elements have pairwise disjoint interiors, every point of the third element (the one corresponding to the third point) lies in $B$, implying that $B$ is not c.w.g $n$, a contradiction.   
\end{proof}

Given an LE $E$, it is an element of some generation $k$ and thus corresponds uniquely to a L{\"u}roth expansion with digits $a_1, \cdots, a_k$.  Since our proof is an intricate and significant extension of the work of the second-named author in~\cite{T4}, we use some ideas and the notation from that paper:  let $R_{a_1 \cdots a_k}:=E$ and $a_1 \cdots a_k$ be referred to as a (finite) string in the letters $\NN \backslash \{1\}$.  For more on strings and the associated ideas from Markov partitions and symbolic dynamics, see Section~2 of~\cite{T4} and Section~\ref{secInfVsFinDistort} of this paper.  The following lemma is easy to verify using elementary properties of L{\"u}roth expansions:
\begin{lemm}\label{lemmSizeofLE}
Given an integer $b \geq 2$ and an LE $R_{a_1 \cdots a_k}$, we have \[\big|\bigcup_{a_{k+1}>b} R_{a_1 \cdots a_k a_{k+1}}\big| = \frac 1 b \big|R_{a_1 \cdots a_k}\big| .\]
\end{lemm}

\section{Schmidt games}\label{secSchmidtGames}

In this section, we define Schmidt games and list their basic properties.  W. Schmidt introduced the games which now bear his name in~\cite{Sch2}.  Let $0 < \a <1$ and $0 < \b<1$.  Let $S$ be a subset of a complete metric space $M$ and $\rho(\cdot)$ denote the radius of a closed ball.  Two players, Player $B$ and Player $A$, alternate choosing nested closed balls \[B_1 \supset A_1 \supset B_2 \supset A_2 \cdots\] on $M$ according to the following rules:  \begin{equation} \label{eqnWinningRules} \rho(A_n) =  \a \rho(B_n) \quad \textrm{ and } \quad \rho(B_n) =  \b \rho(A_{n-1}).\end{equation}  The second player, Player $A$, \textit{wins} if the intersection of these balls lies in $S$.  A set $S$ is called \textit{$(\a, \b)$-winning} if Player $A$ can always win for the given $\a$ and $\b$.  A set $S$ is called \textit{$\a$-winning} if Player $A$ can always win for the given $\a$ and every $\b$.  A set $S$ is called \textit{winning} if it is $\a$-winning for some $\a$.  Schmidt games have three important properties for us~\cite{Sch2}: \medskip

\begin{itemize}
	\item An $\a$-winning set in $\RR^n$ is dense and of full Hausdorff dimension.

\item A countable intersection of $\a$-winning sets is $\a$-winning.

\item An $\a$-winning set in $\RR^n$ with a countable number of points removed is $\a$-winning.
\end{itemize} \medskip

Recently, C.~McMullen introduced a variant of these games in which the rules (\ref{eqnWinningRules}) are replaced by \begin{equation*}\rho(A_n) \geq  \a \rho(B_n) \quad \textrm{ and } \quad \rho(B_n) \geq  \b \rho(A_{n-1}).\end{equation*}   This variant results in \textit{strong winning sets}, and the above properties for winning sets apply, \textit{mutatis mutandis}, to strong winning sets in $\RR^n$.  Such strong winning sets are also winning and behave well under quasisymmetric homeomorphisms~\cite{Mc}.  There have been other recent modifications to these games---see~\cite{KW, Mc, BFKRW}.

\section{Proof of the theorem}

We give the proof for the winning property first and note that this proof also suffices for the strong winning property.  Let $\a := 1/8$ and $0<\b<1$ be arbitrary.   Define $c_1:= 25$ and $b:= \lceil \frac {2 c_1} {\a\b}\rceil$.  We will specify a winning strategy for Player $A$.  It is obvious that Player $A$ can choose a ball that misses any given point of Player $B$'s current ball.  Also, by playing the game for a finite number of rounds, Player $A$ can force all subsequent choices of balls to be less to any given radius.  Thus, without loss of generality, we may assume that $B_1$ does not contain the point $0$ and has diameter strictly less than one.  Hence $B_1$ is c.w.g. $g_1 \geq 1$.  

Let $k :=1$.  We claim that the ball $B_k$ contains at most one accumulation point of generation $g_k-1$.  To see this claim, first note that the only way for $B_k$ to contain an accumulation point of a generation strictly less than $g_k-1$ is for the right endpoint of $B_k$ to be this accumulation point (if $B_k$ contains this accumulation point as some other point, then $B_k$ must contain an LE of generation strictly less than $g_k$, a contradiction).  If $B_k$ contains two accumulation points of generations up to $g_k-1$, then it must contain an LE of generation up to $g_k-1$, a contradiction.  Consequently, $B_k$ contains at most one accumulation point of generations up to $g_k-1$, and, if $B_k$ does contain such an accumulation point and this accumulation point has generation strictly smaller than $g_k-1$, then the right endpoint of $B_k$ is the accumulation point.  

If $B_k$ contains an accumulation point, which we denote by $p_k$, of generations up to $g_k-1$, then we denote \begin{align*}B^r_k &:= \{x \in B_k \mid x \geq p_k\} \\ B^\ell_k&:=\{x \in B_k \mid x \leq p_k\}.\end{align*}  For future choices of Player $B$'s balls (i.e. for integers $k >1$), we make the analogous statements and definitions. We will also handle the case for which $B_k$ contains no such accumulation point---see Case 3 below.

\subsection{Initial step} We intend to use induction.  For handling the initial step, let us define a ball $B_0 \supset B_1$ with radius $\frac b {c_1} \rho(B_1)$---while we only care about the radius of $B_0$, for definiteness, let $B_0$ have the same center as $B_1$.  There are three cases to consider:

\begin{description}
\item [Case 1] \textit{The ball $B_1$ contains an accumulation point of generations up to $g_1-1$, and $|B^r_1| \geq |B^\ell_1|.$}

Let $p_1$ denote the accumulation point.  Since it is not the right endpoint of $B_1$, it must be (as shown above) of generation $g_1-1$.  Let $R_\g$ be the LE of the same generation as $p_1$ and having $p_1$ as left endpoint (therefore $p_1$ is the accumulation point of $R_\g$).  Let $q$ denote the left endpoint of $R_{\g2}$ and $p_1^+$ denote the accumulation point of generations up to $g_1-1$ immediately to the right of $p_1$.  Note that $p^+_1$ always exists because $0$ is identified with $1$ and $1$ is a point to the right of all points in $B_1$, a ball which, recall, does not contain $1$.  Then $B^r_1$ is properly contained in the closed interval $[p_1, p^+_1]$.  

Consider two subcases.  First let $B^r_1$ be properly contained in the interval $[p_1, q)$.  It is now clear that Player $A$ can choose $A_1$ to be contained in $B^r_1$ and disjoint from the closed ball $\bar{B}(p_1, \frac 1 b \rho(B_0))$, the LE $R_{\g2}$, and every element of $G_{g_1-1} \backslash \{R_\g\}$.

The other subcase is the case in which $B^r_1$ is not properly contained in the interval $[p_1, q)$.  By Lemma~\ref{lemmSizeofLE}, it follows that $|B^r_1| \geq \frac 1 2 |R_{\g}|$.  Also, since $B_1$ is contained in the union of $R_\g$ and its left-adjacent LE of the same generation (the left-adjacent LE always exists), then $2 |R_{\g}| >|B_1|$ also follows by the same lemma.  The conclusion of the previous paragraph is immediate.

Now $B_2$ is chosen.  If $B_2$ is c.w.g $g_1$, then Player $A$ may choose any ball $A_2$ allowed by the Schmidt game.  Player $A$ can continue to play in this way for any $B_k$ c.w.g. $g_1$.  We claim, however, that, at some iterate $n \geq 2$ of the game, $B_n$ will be commensurate with a generation strictly greater than $g_1$.  The claim follows because, by choice of $A_1$, $A_1$ can contain only a finite number of LEs of generation $g_1$.  If $A_1$ contains no such LE, then the claim follows for $B_2$.  If $A_1$ does contain such an LE, then it contains one of least length.  Thus, for some iterate $n\geq 2$, $B_n$ is too small to contain any such LE, which implies the claim.  

Let $j_2 \in \NN$ such that $B_{j_2}$ is c.w.g $g_1$ and $B_{1+j_2}$ is c.w.g. $g_2 > g_1$.  By the above, $B_{1+j_2} \subset R_{\g} \backslash R_{\g2}$.  Form the set of LEs of the relevant generations that intersect $B_{1+j_2}$:  \[\mE := \big\{E \in \bigcup_{g = g_1}^{g_2-1} G_g \mid E \cap B_{1+j_2} \neq \emptyset \big\}.\]  Since $B_{1+j_2}$ is c.w.g. $g_2$, the ball is contained in at most two elements of $G_{g_2-1}$ by Corollary~\ref{coroAtMost1AccPoint}.  Therefore, every element in $\mE$ must contain one or the other element of generation $g_2-1$.\footnote{Even more, every element of $\mE$ not of generation $g_2-1$ must contain both (if there are two).}  We prove the following lemma, whose analogous version is also needed for the inductive step.

\begin{lemm}\label{lemmBallBiggerHoles}
The diameter of the ball $B_{1+j_2}$ is larger than the diameter of any interval $(p, p+\frac{\sqrt{c_1}}{b} |E|)$ where $E \in \mE$ and $p$ is the accumulation point of $E$.
\end{lemm}

\begin{proof}
Since  $B_{1+j_2}$ is c.w.g. $g_2$, there exists an element $E_{g_2} \in G_{g_2}$ such that $E_{g_2} \subset B_{1+j_2}$.  Consequently, there exists a unique element $E_{g_2-1} \in G_{g_2-1}$ such that $E_{g_2} \subset E_{g_2-1}$.  By recursion, we may define the chain of inclusions\[E_{g_2} \subset E_{g_2-1} \subset \cdots \subset E_{g_1}\] where $E_{g} \in G_g$.

Since $B_{1+j_2} \subset A_1$, we have that $E_{g_1} \cap A_1 \neq \emptyset$.  By the choice of $A_1$, the L{\"u}roth expansion corresponding to $E_{g_1}$ cannot end in the digit $2$, which implies that both $E^+_{g_1}$ and $E^-_{g_1}$ exist.

Since $B_{1+j_2} \subset B_{j_2}$, we have that $E_{g_1} \cap B_{j_2} \neq \emptyset$.  Since $B_{j_2}$ is c.w.g $g_1$, $B_{j_2}$ must completely contain (at least) one element of the set $\{E^-_{g_1}, E_{g_1}, E^+_{g_1}\}$.

Let $\tilde{B}_{j_2}$ denote the closed ball with the same center as $B_{j_2}$, but with $\sqrt{c_1}$ times the radius.  By an easy argument using Lemma~\ref{lemmSizeofLE}, we have that $E_{g_1} \subset \tilde{B}_{j_2}$.  Now note that $\sqrt{c_1} |B_{1+j_2}| \geq \sqrt{c_1} \frac{2c_1}{b} |B_{j_2}| \geq  \frac{c_1}{b} |E_g|$ for $g_1 \leq g \leq g_2$, which implies the desired result. \end{proof}

\begin{coro}\label{coroDisjointOnLowerGens}
Let $g_2 > g_1+1$.  The ball $B_{1+j_2}$ is disjoint from every interval $(p, p+\frac{1}{b} |E|)$ where $E$ is an LE of generations from $g_1$  to $g_2-2$ and $p$ is the accumulation point of $E$.
\end{coro}

\begin{proof}
 
Let $F$ and $F^-$ be the (possibly) two LEs of generation $g_2-1$ whose union contains (properly) $B_{1+j_2}$ (by Corollary~\ref{coroAtMost1AccPoint}).  Since $F$ and $F^-$ are adjacent LEs, they lie in the same LE of generation $g_2-2$ (if not, then $B_{1+j_2}$ is not c.w.g. $g_2$, which is a contradiction) and thus the same LE---call it $F'$---of whatever generation $E$ is.  There are two cases.
\begin{description}
\item[Case A]  The LEs $E$ and $F'$ are distinct elements of the same generation.

The LEs $E$ and $F'$ can only intersect at an endpoint, which is an accumulation point of either $E$ or $F'$.  However, since $B_{1+j_2}$ is  contained in $F'$, it is disjoint from the interior of $E$, implying the desired result.

\item[Case B] The LEs $E$ and $F'$ are the same element.  

Assume that the conclusion does not hold.  Let $E:=R_\g$ for some string $\g$.  Now by the elementary properties of L{\"u}roth expansions (\cite{DK}), $|R_{\g b}| = \frac{1}{b(b-1)} |R_\g|$---note that the LE $R_{\g b}$ is disjoint from $[p, p+\frac{1}{b} |E|)$ and $[p, p+\frac{1}{b-1} |E|] = [p, p+\frac{1}{b} |E|) \cup R_{\g b}$ by elementary L{\"u}roth expansion properties (see also Lemma~\ref{lemmSizeofLE}).  Since $p$ is the left endpoint of $E$ and since $B_{1+j_2}$ is contained in $E$, then the assumption implies that the left endpoint of $B_{1+j_2}$ is contained in $[p, p+\frac{1}{b} |E|)$.  Therefore, Lemma~\ref{lemmBallBiggerHoles} implies that  $B_{1+j_2}$ contains $R_{\g b}$, an LE of generation at most $g_2-1$.  This contradicts the fact that $B_{1+j_2}$ is c.w.g. $g_2$.
\end{description} \end{proof}

If the interior of $B_{1+j_2}$ contains an accumulation point $q$ of generations up to $g_2-1$, then it must be of generation $g_2-1$ (otherwise, $B_{1+j_2}$ is not c.w.g. $g_2$).  Let $E_q$ be the LE corresponding to $q$.  Let $q^-$ be the accumulation point of $E_q^-$.  Then the only two intervals that Player $A$ must avoid and may (possibly) intersect $B_{1+j_2}$ are $(q^-, q^- +\frac{1}{b} |E^-_{q}|)$ and $(q, q +\frac{1}{b} |E_{q}|)$.   If  the interior of $B_{1+j_2}$ does not contain an accumulation point of generations up to $g_2-1$, then $B_{1+j_2}$ lies completely in some $E_q \in G_{g_2-1}$ where $q$ is the corresponding accumulation point---and the one interval that Player $A$ must avoid is $(q, q +\frac{1}{b} |E_{q}|)$.   In either case, Lemma~\ref{lemmBallBiggerHoles} implies that $B_{1+j_2}$ is at least $\sqrt{c_1}/2$ times larger than the (union of) interval(s) Player $A$ must avoid, and the corollary implies that if $g_2 > g_1+1$, then $B_{1+j_2}$ automatically avoids any intervals $(p, p +\frac{1}{b} |E|)$ where $E \in \bigcup_{g=g_1}^{g_2-2} G_g$. \medskip

\item [Case 2] \textit{The ball $B_1$ contains an accumulation point of generations up to $g_1-1$, and $|B^r_1| < |B^\ell_1|.$}

Let $p_1$ denote the accumulation point.  Since $B_1$ is c.w.g $g_1$, $B^\ell_1$ is properly contained in an element $R_\g \in G_{g_1-1}$ by Corollary~\ref{coroAtMost1AccPoint}.  Then $p_1$ is the right endpoint of $R_\g$.  Player $A$ chooses $A_1$ to be the closed ball with right endpoint $p_1$.  Since $|B^\ell_1| < |R_\g|$ and $\a$ is small enough, $A_1 \subset R_{\g2}$ and, moreover, if $p$ is the accumulation point of $R_{\g2}$, then $A_1$ is disjoint from $[p, p +\frac{1}{b} |R_{\g2}|)$.  Therefore, $A_1$ is disjoint from all intervals $(p, p+\frac{1}{b} |E|)$ where $p$ is an accumulation point of generations up to $g_1$ and $E$ is the LE corresponding to $p$.

Now $B_2$ is chosen.  By the choice of $A_1$,  $A_1$ is c.w.g $g > g_1$, and, even more precisely, $R_{\g2\tilde{\g} 2} \subset A_1\subset R_{\g2\tilde{\g}}$ where $\tilde{\g}$ is a string of $g-g_1-1$ repeating digits $2$.  Thus, we have that $|A_1| > |R_{\g2\tilde{\g} 2}| = \frac 1 2 |R_{\g2\tilde{\g}}|$ by Lemma~\ref{lemmSizeofLE}.  Let $B_2$ be c.w.g $g_2$---hence $g_2 \geq g$.  If $g_2 = g$, then $B_2$ is contained in $E:=R_{\g2\tilde{\g}}$, an LE of generation $g_2-1$.    If $g_2 > g$, then, by Corollary~\ref{coroAtMost1AccPoint}, $B_2$ must be contained in at most two adjacent LE of generation $g_2-1$ both of which lie in $R_{\g2\tilde{\g}}$; denote these two elements by $E^-$ and $E$.  

Thus, for $g_2 \geq g$, we have $|E^-| < |E| \leq |R_{\g2\tilde{\g}}|$.  Consequently, $|B_2| \geq \frac {2 c_1} {\a b} |A_1| \geq \frac { c_1} {\a b} |R_{\g2\tilde{\g}}|$, which is much larger than the (at most two) interval(s)---namely, $[q^-, q^-+ \frac 1 b |E^-|)$ and $[q, q+ \frac 1 b |E|)$ where $q^-$ and $q$ are the accumulation points of $E^-$ and $E$, respectively---that Player $A$ must avoid.  

\begin{lemm}
Let $g_2 > g_1+1$.  The ball $B_{2}$ is disjoint from every interval $(p, p+\frac{1}{b} |F|)$ where $F$ is an LE of generations from $g_1$  to $g_2-2$ and $p$ is the accumulation point of $F$.
\end{lemm}

\begin{proof}
Since $E$ and $E^-$ are adjacent LEs of generation $g_2-1$, they lie in the same LE of generation $g_2-2$ and thus the same LE---call it $E'$---of whatever generation $F$ is.  There are two cases.
\begin{description}
\item[Case A]  The LEs $F$ and $E'$ are distinct elements of the same generation.

The proof is analogous to the proof of Case A of Corollary~\ref{coroDisjointOnLowerGens}.

\item[Case B] The LEs $F$ and $E'$ are the same element.

There are two subcases to consider.  The first is when $F$ is of generations $g_1$ to $g-2$ (provided that $g>g_1+1$; otherwise, this subcase is not needed).  Since $A_1\subset R_{\g2\tilde{\g}}$ where $\tilde{\g}$ is a string of $g-g_1-1$ repeating digits $2$, the result is clear from the elementary properties of L{\"u}roth expansions and the size of $b$. 

The second subcase is when $F$ is of generations $g -1$ to $g_2-2$.  Assume that the conclusion does not hold.  The condition of this case, Case B, implies that $F \subset R_{\g2\tilde{\g}}$.  Let $\t$ be the (possibly empty) string such that $F = R_{\g2\tilde{\g} \t}$ and $p$ be the left endpoint of $F$.  Thus $|R_{\g2\tilde{\g} \t b}| = \frac{1}{b(b-1)} |R_{\g2\tilde{\g}\t}|$---note that the LE $R_{\g2\tilde{\g}\t b}$ is disjoint from $[p, p+\frac{1}{b} |F|)$ and $[p, p+\frac{1}{b-1} |F|] = [p, p+\frac{1}{b} |F|) \cup R_{\g2\tilde{\g}\t b}$.  Since $p$ is the left endpoint of $F$ and since $B_2$ is contained in $F$, then the assumption implies that the left endpoint of $B_{2}$ is contained in $[p, p+\frac{1}{b} |F|)$.  Since $|B_2| \geq \frac {2 c_1} {\a b} |A_1| \geq \frac { c_1} {\a b} |R_{\g2\tilde{\g}}| \geq \frac { c_1} {\a b}|F|$, we have that $B_{2}$ contains $R_{\g2\tilde{\g}\t b}$, an LE of generation at most $g_2-1$.  This contradicts the fact that $B_{2}$ is c.w.g. $g_2$.
\end{description}\end{proof}

Let $j_2=1$; then, we conclude as in Case 1.

\item [Case 3] \textit{The ball $B_1$ does not contain an accumulation point of generations up to $g_1-1$.}

The given condition on $B_1$ implies that it is completely contained in the interior of an LE $E$ of generation $g_1-1$.  Now since $B_1$ is c.w.g. $g_1$, it completely contains an LE of generation $g_1$---if it were to contain such an element with last digit $2$, it would contain an accumulation point of generations up to $g_1-1$, a contradiction.  Therefore, it contains an LE of generation $g_1$ with last digit not $2$.

Let $p_1$ be the accumulation point corresponding to $E$.  And let $p_1^+$ be the accumulation point of generations up to $g_1-1$ which is also the right endpoint of $E$.

Setting $E = R_\g$ and treating $B_1$ as we did $B_1^r$ in Case 1, we handle this case exactly as Case 1, except for the following.  Recall the definition of $q$ from Case 1.  The first subcase is handled exactly as in Case 1.  The second subcase---in which $B_1$ is not properly contained in the interval $[p_1, q)$---is handled as follows.  Since $B_1$ is c.w.g. $g_1$, it cannot be properly contained in $R_{\g2}$.  Therefore, the condition of this subcase implies that $B_1$ contains $q$.  Whence, it must contain $R_{\g3}$ because it is c.w.g $g_1$ and it cannot contain $R_{\g2}$.  Therefore, $B_1$ is large relative to $R_\g$ (more precisely, it is at least $1/6$ the length of $R_\g$) and it is clear that $A_1$ can be chosen to be disjoint from the closed ball $\bar{B}(p_1, \frac 1 b \rho(B_0))$, the LE $R_{\g2}$, and every element of $G_{g_1-1} \backslash \{R_\g\}$.  The remainder of this case (Case 3) is handled exactly as the remainder of Case 1.

\end{description}
This completes the three cases and the initial step of  the induction.

\subsection{Induction step}  Our induction index is $n$.  Let $j_1 := 1$.  Let $J_n := \sum_{i=1}^n j_i$.  By the induction hypothesis, $B_{J_n}$ is c.w.g $g_n > g_{n-1}$; therefore, there exists (by Corollary~\ref{coroAtMost1AccPoint}), either two LEs $E$ and $E^-$ of generation $g_n-1$ such that $B_{J_n} \subset E^- \cup E$ or one LE $E$ of generation $g_n-1$ such that $B_{J_n} \subset E$.  If $g_n > g_1 +1$, then, also by the induction hypothesis, $B_{J_n}$ avoids $(q, q+\frac{1}{b} |F|)$ where $q$ is an accumulation point of generations $g_1$ up to $g_{n}-2$ and $F$ is the LE corresponding to $q$.  Finally, again by the induction hypothesis, $|B_{J_n}|$ is (at least) $\sqrt{c_1}/2$ times larger than $|(p^-, p^-+\frac{1}{b} |E^-|) \cup (p, p+\frac{1}{b} |E|)|$ where $p$ is the accumulation point of $E$ and $p^-$ is the accumulation point of $E^-$.

Thus, it is possible for Player $A$ to chose $A_{J_n}$ to be disjoint from all $(p, p+\frac{1}{b} |E|)|$ where $E$ is an LE of generation $g_n-1$ and $p$ is the accumulation point of $E$.  We choose $A_{J_n}$, which will be so disjoint, by considering three cases:

\begin{description}
	\item[Case 1] \textit{The ball $B_{J_n}$ contains an accumulation point of generations up to $g_n-1$, and $|B^r_{J_n}| \geq |B^\ell_{J_n}|.$}

This case is handled in the analogous way to Case 1 of the initial step---except $B_1$ is replaced by $B_{J_n}$, $A_1$ is replaced by $A_{J_n}$, $B_2$ is replaced by $B_{J_n+1}$, and so on.  Also, the two subcases give similar conclusions as in the inductive step, except $A_{J_n}$ is disjoint from $(p_n, p_n+\frac{1}{b} |R_\g|)$ and the LE $R_{\g2}$ (recall that $p_n$ must be an accumulation point of generation $g_n-1$ for this case).  The rest of this case is analogous to the initial step.

	\item[Case 2] \textit{The ball $B_{J_n}$ contains an accumulation point of generations up to $g_n-1$, and $|B^r_{J_n}| < |B^\ell_{J_n}|.$}

	This case is handled in the analogous way to Case 2 of the initial step---except $B_1$ is replaced by $B_{J_n}$, $A_1$ is replaced by $A_{J_n}$, $B_2$ is replaced by $B_{J_n+1}$, and so on.  

	\item[Case 3] \textit{The ball $B_{J_n}$ does not contain an accumulation point of generations up to $g_n-1$.}

This case is handled in the analogous way to Case 3 of the initial step---except $B_1$ is replaced by $B_{J_n}$, $A_1$ is replaced by $A_{J_n}$, $B_2$ is replaced by $B_{J_n+1}$, and so on.  (Note that, as in Case 1 of the induction step, the two subcases follow easily.  For the first subcase, consider as in Case 1 of the induction step; for the second subcase, consider as in Case 3 of the initial step.) \end{description}

This completes the induction.  

\subsection{Finishing the proof for winning} Let $x \in X$.  If $E$ is an LE, let $p_E$ be the accumulation point of $E$.  The induction above shows that the set \begin{align*}BL(b):=\bigg\{x & \in X \mid \text{ there exists an } n(x) \in \NN \text{ such that } \\ & x\notin \bigcup_{m \geq n}  \bigcup_{E \in G_m} (p_E, p_E+ \frac 1 b |E|) \text{ and there exists } c(x) >0 \text{ such that } \\ & x \notin \bigcup_{m=0}^{n-1} \bigcup_{E \in G_m} (p_E, p_E+ c |E|)  \bigg\}\end{align*} is $(\a,\b)$-winning.  (Note that a given point $x$ lies in at most two LEs of the same generation.  Therefore, for a fixed $n(x) \in \NN$, $x$ lies in either $n+1$ or $n+2$ LEs of generations up to $n$.)  Consequently, the set $BL:=\bigcup_{a \geq b} BL(b)$ is $\a$-winning.  

Now let $x \in BL$.  Then $x \notin \bigcup_{m \geq 0} \bigcup_{E \in G_m} (p_E, p_E+ \tilde{c} |E|)$ where $\tilde{c} := \min \{c, \frac 1 a\}$.  Thus, $BL$ is contained in the set of numbers with bounded L{\"u}roth expansion.  And the proof of the theorem for $\a$-winning is complete.

Finally, note that accumulation points have two L{\"u}roth expansions (using a slightly less strict algorithm than the one given in Section~\ref{secIntro}).  One expansion is finite and the other has trailing digits $2$---in either case, bounded.  Removing all accumulation points from $BL$, however, still results in an $\a$-winning set.

\subsection{Proof for strong winning}\label{secProofForSWin}  Recall that we must find a winning strategy for Player $A$.  Therefore, we may and do constrain Player $A$'s choices of radius size to be always equal to $\a$ times the radius of Player $B$'s current ball.  With this constraint, the proceeding proof, without change, shows that $BL$ is $\a$-strong winning.

\section{Conclusion}\label{secConcl}

In this section, we compare the result in this paper to the second-named author's result in~\cite{T4}, discuss infinite versus bounded distortion, and suggest further applications of the technique in this paper.

\subsection{Infinite versus finite element Markov partitions}\label{secInfVsFinDistort}

To distinguish between the result in this paper and the work in Section~5 of~\cite{T4}, it helps to consider strings and their associated Markov partitions.  The basic idea is that the strings encode the motion of the associated dynamical system (in our case the L{\"u}roth map $T$), and this encoding is established using Markov partitions.  In our case, the collection of L{\"u}roth elements in $G_1$ is a natural choice for a Markov partition associated with $T$.  In particular, $G_1$ satisfies the definition of a Markov partition for $T$ as stated in Section~2.1 of~\cite{T4}---except that the cardinality of $G_1$ is not finite, but infinite.  Note also that the definition of $R_\a$ for a (finite) string $\a$ and properties~(2.6 -- 2.13) of~\cite{T4} also hold for $G_1$.  (For completeness, the transition matrix for our $T$ is an infinite matrix of all $1$'s in the letters $\NN \backslash \{1\}$.)  Hence, $G_1$ is an example of an infinite element Markov partition, while the Markov partitions in~\cite{T4} are the usual (finite element) Markov partitions.

But, although $G_1$ is an infinite element Markov partition, it is, perhaps, the simplest example.  It and its associated L{\"u}roth expansions are the infinite element analog of $n$-ary expansions where $n>1$ is an integer.  In this latter case, we can take the simplest Markov partition of $\{[0, \frac 1 n], \cdots, [\frac {n-1} n, 1]\}$.  If, further, we only ask to miss an open neighborhood with $0$ as the left endpoint, then the proof in~\cite{T4} greatly simplifies:  the difficulty in this proof is in dealing with missing a neighborhood of an arbitrary point and in dealing with the nonlinearity of the dynamical system.\footnote{As far as the authors' know, the nonlinear case of this type of result in higher dimensions is still open.}  Moreover, missing a neighborhood of $0$ for the $n$-ary expansion was already shown by Schmidt in 1966~\cite{Sch2}.

\subsection{Infinite versus bounded distortion}\label{secInfVsFinDistort2}  What is different between this infinite element Markov partition and the usual finite element Markov partitions is distortion.  In particular, finite element Markov partitions have the {\it bounded distortion property}, or, equivalently, have {\it bounded distortion (in the sense of Markov partitions)}, in that, over a fixed, finite window of generations, the ratio of length (or other relevant measure) of the element with largest length to smallest length is finite (and this finite bound depends only on the size of the window).  Explicitly, the bounded distortion property is the property given in Lemma~2.2 of~\cite{T4}.   Finite element Markov partitions satisfy this property (as the lemma asserts), while our infinite element Markov partition clearly does not, even if the window is the smallest that it can be, namely over the same generation.\footnote{It seems appropriate to refer to this property as \textit{infinite distortion}, while leaving \textit{unbounded distortion} for the case where, as one slides the fixed window toward infinity, the bound on distortion grows to infinity.}  Thus, our dynamical system has \textit{infinite distortion (in the sense of Markov partitions)}, while the ones considered in~\cite{T4} have bounded distortion.  And this difference does not allow us to port the proofs in~\cite{T4}, even in the simplest case of missing an open neighborhood with left endpoint $0$---the infinite distortion, which manifests itself in our notion of accumulation points for the L{\"u}roth expansion, is a new source of infinity and a new source of difficulties.  But, the infinite distortion is also a new source of solutions in that, the accumulation points, which manifest the infinite distortion, are also used, via the proofs of Lemma~\ref{lemmBallBiggerHoles} and Corollary~\ref{coroDisjointOnLowerGens}, to our advantage because these points can be distinguished from non-accumulation points via the geometry of the expansion and the notion of commensurate.

\subsection{Further applications}\label{subsecFA}

We intend the proof in this paper to be a model for applying Schmidt games to cases where infinite distortion exist.  These cases naturally occur in number theory and dynamical systems, especially in other types of expansions.  Many further applications are possible.

We mention one such application.  The continued fraction expansion is the most natural expansion because it leads to notions of best approximations---the applications of continued fractions are far-reaching and important.  Like its variant the L{\"u}roth expansion, the digits of the continued fraction expansion are generated by a dynamical system, in this case the well-known Gauss map, which is a system where infinite distortion exists, much like our map $T$.  Applying our proof technique, but modifying it to handle the fact that accumulation points for the Gauss map alternate between left and right endpoints of the ``Gauss elements'' for odd and even generation numbers, should lead to another proof that the set of real numbers with bounded continued fraction expansion or, equivalently, the set of badly approximable numbers is winning.  The two known proofs are based on the repulsion of the elements of a Farey sequence for fixed denominators~\cite{Sch2}, which is a number-theoretic proof, and bounded orbits under the geodesic flow in the space of unimodular lattices~\cite{Da3}, which is a dynamical proof.  The proof involving the flow, however, is not the most elementary dynamical proof because the geodesic flow can be regarded as a suspension of the Gauss map and thus is not the dynamical system which defines the continued fraction expansion, but an induced system.  A proof adapting our technique in this paper would just involve the Gauss map and be an elementary dynamical proof.

\subsection*{Acknowledgments}  The authors would like to thank Florin Boca, Kevin Ford, and Dmitry Kleinbock for their helpful comments.

\end{document}